\documentclass[11pt]{article}
\usepackage[a4paper,margin=2.5cm]{geometry}
\usepackage{amsmath,amssymb,amsthm,enumitem}
\usepackage{xcolor}
\usepackage[
    colorlinks=true,
    linkcolor=blue,   
    citecolor=blue,   
    urlcolor=blue!60 
]{hyperref}

\newtheorem*{theorem*}{Theorem}

\newtheorem{lemma}{Lemma}
\newtheorem{corollary}{Corollary}

\DeclareMathOperator{\Cay}{Cay}
\DeclareMathOperator{\SubF}{SubF}

\setlist[enumerate]{%
  labelwidth=\parindent,
  labelindent=\parindent,
  leftmargin=0pt,
  labelsep=*,
  align=left,
  itemindent=\dimexpr\parindent+1.5em\relax,
  itemsep=2pt,
  topsep=2pt,
  partopsep=0pt,
  parsep=0pt
}

\title{Factors in infinite groups}
\author{Mikhail Kabenyuk}
\date{}

\begin{document}
\maketitle
\begin{abstract}
Let $G$ be a group and $A\subseteq G$ a non-empty subset. A right $s$-factor associated with $A$ is a maximal subset
 $U\subseteq G$ such that the product $AU$ is direct. 
The lower and upper $s$-indices $|G:A|^-$ and $|G:A|^+$ are defined as the minimum and the supremum of the cardinalities of such maximal sets $U$.
The subset $A$ is called stable if $|G:A|^- = |G:A|^+$, and $G$ is called stable if every subset of $G$ is stable.

Using a graph-theoretic reformulation in terms of Cayley graphs, we prove that every infinite group is unstable. 
Equivalently, for every infinite group $G$ there exists a subset $A\subseteq G$ for which maximal subsets $U$ with 
direct product $AU$ do not all have the same cardinality. This gives a negative answer to Question 21.58 of the 
Kourovka Notebook.
\end{abstract}

\section{Introduction}
Let $G$ be a group and let $A$ be a fixed non-empty subset of $G$.
A non-empty subset $U\subset G$ is called a \textit{right $s$-factor of $G$ associated with $A$}
if 
\begin{enumerate}
    \item[(SF1)] every element $x\in AU$ can be written uniquely as $x=au$ with $a\in A$ and $u\in U$
and 
    \item[(SF2)] $U$ is maximal (with respect to inclusion) with this property.
\end{enumerate}
Right $s$-factors exist for every non-empty subset of $A\subseteq G$ (by Zorn’s lemma).
Let $\SubF_r(G\!:\!A)$ be the set all right $s$-factor of $G$ associated with $A$.
Also,
\begin{align*}
&|G:A|^-=\inf\{|U|\mid U\in\SubF_r(G\!:\!A)\},\\
&|G:A|^+=\sup\{|U|\mid U\in\SubF_r(G\!:\!A)\}.
\end{align*}
A subset $A$ is called stable if $|G:A|^-=|G:A|^+$.
A group $G$ is called stable if every subset of $G$ is stable.
Note that for any $x\in G$, the sets $A$ and $xA$ have the same collection of right $s$-factors,
and thus $|G:A|^{\pm}=|G:xA|^{\pm}$.
Hence in many arguments we may assume $e\in A$.

\medskip
These notions (under the names ``subfactor'' and ``index stability'') were introduced in \cite{Hooshmand2020} and \cite{Hooshmand2023}. 
The term $s$-factor was introduced in \cite{Kab2026}.
In \cite{Kab2026}, a connection between $s$-factors and maximal independent sets in graphs was established as follows. For a given subset $A$ with $e\in A$, define
\[
\partial A=A^{-1}A\setminus\{e\}=\{a^{-1}b\mid a,b\in A,\ a\neq b\}.
\]
Clearly, the set $\partial A$ is symmetric (that is, together with each element it contains its inverse) 
and does not contain $e$ by construction.
The (left) Cayley graph $\Cay(G,\partial A)$ is defined as follows: its vertices are the elements of $G$, 
and two vertices $g,h$ are adjacent if $gh^{-1}\in \partial A$.
The graph $\Cay(G,\partial A)$ is simple, i.e., it is undirected and has no loops or multiple edges.

If the product $AU$ satisfies condition (SF1), then $U$ is an independent set in $\Cay(G,\partial A)$.
Indeed, if $u\neq v$ and $u\sim v$, then $uv^{-1}\in\partial A$, that is, $uv^{-1}=a^{-1}b$ for some $a,b\in A$,
or equivalently $au=bv$. By (SF1), this implies $a=b$ and $u=v$, a contradiction.

Conversely, if $U$ is an independent set in $\Cay(G,\partial A)$, then $AU$ satisfies (SF1).
Indeed, if $au=bv$ for some $a,b\in A$ and $u,v\in U$, then
\[
uv^{-1}=a^{-1}b.
\]
If $u\neq v$, then (by independence) $uv^{-1}\notin\partial A$, whereas $a^{-1}b\in\partial A$ whenever $a\neq b$.
Hence $u=v$, and then $a=b$.

If, in addition, condition (SF2) holds, then $U$ is a maximal independent set in $\Cay(G,\partial A)$.
Indeed, if $x\in G\setminus U$, then by (SF2) there exist $a,b\in A$ and $u\in U$ such that $ax=bu$.
If $a=b$, then $x=u\in U$, a contradiction; hence $a\neq b$, so
\[
xu^{-1}=a^{-1}b\in\partial A,
\]
and therefore $x\sim u$. Thus $U\cup\{x\}$ is not independent.

Conversely, if $U$ is a maximal independent set in $\Cay(G,\partial A)$, then (SF2) holds.
Indeed, for any $x\in G\setminus U$, maximality implies that $U\cup\{x\}$ is not independent, so there exists
$u\in U$ such that $x\sim u$. Hence
\[
xu^{-1}\in\partial A,
\]
so $xu^{-1}=a^{-1}b$ for some $a,b\in A$, equivalently $ax=bu$.

Therefore, $\SubF_r(G\!:\!A)$ is precisely the set of all maximal independent sets of the graph
$\Gamma=\Cay(G,\partial A)$. Hence $|G:A|^+=\alpha(\Gamma)$, the independence number of $\Gamma$
(we use this term even when it is an infinite cardinal), and $|G:A|^-=i(\Gamma)$, the independent domination number.
The graph $\Gamma$ is \textit{well-covered} if every maximal independent set in $\Gamma$ is maximum,
or equivalently, $\alpha(\Gamma)=i(\Gamma)$.

In this language, we can give another definition of a stable subset.
A subset $A\subset G$ with $e\in A$ is called stable if the Cayley graph $\Cay(G,\partial A)$ is well-covered.

In \cite[Theorem 3.1]{HooshmandArani} (see also \cite[Theorem 16]{Kab2026}), 
it was proved that among finite groups there are only 14 stable groups; 
all other finite groups are unstable.
In this note we prove that every infinite group is unstable.

This last statement answers Question~21.58 from the Kourovka Notebook \cite{Kourovka2026}:

21.58. We say that a product $XY=\{xy\mid x\in X,\ y\in Y\}$ of two subsets $X,Y$ of
a group $G$ is direct if for every $z\in XY$ there are unique $x\in X$, $y\in Y$ such that
$z=xy$. Is there an infinite group $G$ such that every subset $A\subset G$ satisfies the
following property: all maximal subsets $B$ for which the product $AB$ is direct
have the same cardinality?

The stated theorem implies that no such infinite group exists.

\medskip
\noindent\textbf{Notation.} 
We use the following notation from group theory and graph theory.
The symbol $e$ denotes the identity element of a group.
If $A$ and $B$ are subsets of a group, then
\[
A^{-1}=\{a^{-1}\mid a\in A\}\quad\text{and}\quad
AB=\{ab\mid a\in A,\;b\in B\}.
\]
For a subgroup $H\le G$, we write $H^* = H\setminus\{e\}$.
If $v$ is a vertex of a graph, then $N(v)$ denotes the set of all neighbors of $v$, 
and $N[v] = N(v)\cup\{v\}$.
If $u$ and $v$ are vertices of a graph, then we write $u \sim v$ if and only if $u$ and $v$ are adjacent.
The notation $X\subset Y$ does not exclude the case where $X = Y$.

All undefined graph-theoretic terms are used in the sense of \cite{Berge}, and all undefined group-theoretic terms 
are used in the sense of \cite{Robinson}.

\section{Graph-theoretic reformulation}

Along with the graph $\Gamma=\Cay(G,\partial A)$, in our setting it will be useful to consider two more graphs.
Let $F=G\setminus A^{-1}A$. Then $F$ is also symmetric and does not contain $e$, so we may consider
the Cayley graph $\Cay(G,F)$. It is easy to see that this graph is the complement of $\Gamma$,
that is, distinct $g,h\in G$ form an edge in $\Cay(G,F)$ if and only if 
$g,h$ are nonadjacent in $\Gamma$.
We shall denote this by
\[
\Cay(G,F)=\overline{\Gamma}.
\]
Since $\Gamma$ and $\Cay(G,F)$ are complementary graphs on the same vertex set $G$, 
independent sets in $\Gamma$ are exactly cliques in $\Cay(G,F)$, and maximal independent sets in $\Gamma$ 
are exactly maximal cliques in $\Cay(G,F)$. Hence $\alpha(\Gamma)$ is the cardinality of a largest clique in 
$\Cay(G,F)$, while $i(\Gamma)$ is the minimum cardinality of a maximal clique in $\Cay(G,F)$.

We also define another useful graph.
Let $F$ be a symmetric subset of $G$ with $e\notin F$. Define the graph $\Delta=\Delta(F)$ by:
the vertices of $\Delta$ are the elements of $F$, 
and for $u,v\in F$, we have $u\sim v$ in $\Delta$ if and only if $uv^{-1}\in F$.
It is easy to see that if $F=G\setminus A^{-1}A$ for some $A\subset G$, 
then $\Delta(F)$ is the subgraph induced by $F$
in the graph $\Cay(G,F)$.
A \emph{clique} in $\Delta$ is a set of pairwise adjacent vertices.
For a finite graph $\Delta$, define
\[
\omega(\Delta)=\max\{|Q|\ \mid Q \text{ is a clique in }\Delta\},
\quad
\iota(\Delta)=\min\{|Q|\ \mid Q \text{ is a maximal clique in }\Delta\}.
\]

\begin{lemma}\label{lemma:DeltaF_cliques}
Let $A\subset G$, let $F=G\setminus A^{-1}A$, and let $\Gamma=\Cay(G,\partial A)$. 
Assume that $F$ is finite and $\Delta=\Delta(F)$.
Then $\omega(\Delta)=\alpha(\Gamma)-1$, 
and $\iota(\Delta)=i(\Gamma)-1$.

Consequently, the graph $\Gamma$ is well-covered if and only if
all maximal cliques of $\Delta(F)$ have the same size.
\end{lemma}

\begin{proof}
A subset $Q\subseteq F$ is a clique in $\Delta(F)$ if and only if $\{e\}\cup Q$ is a clique in $\Cay(G,F)$.
Hence maximal cliques in $\Delta(F)$ correspond exactly to maximal cliques in $\Cay(G,F)$ that contain $e$, and their sizes differ by $1$.

Now let $K$ be any clique in $\Cay(G,F)$. Choosing $x\in K$ and applying the left translation $g\mapsto x^{-1}g$,
we obtain a clique of the same size containing $e$. Since left translations are automorphisms of $\Cay(G,F)$,
they preserve maximality as well. Therefore, the size of a largest clique in $\Cay(G,F)$ is one more than the size
of a largest clique in $\Delta(F)$, and the size of a smallest maximal clique 
in $\Cay(G,F)$ is one more than the size of a smallest maximal clique in $\Delta(F)$.

Since $\Cay(G,F)=\overline{\Gamma}$, cliques in $\Cay(G,F)$ are exactly independent sets in $\Gamma$,
and maximal cliques in $\Cay(G,F)$ are exactly maximal independent sets in $\Gamma$. 
Thus the stated equalities follow, and the final claim is immediate.
\end{proof}

\section{A lemma on finite symmetric subsets}

We will use a fact stated in \cite{Hooshmand2019}
for the additive group of rational numbers and proved for an arbitrary infinite group by Ravsky in 
\cite{Ravsky2025}.

\begin{lemma}\label{lemma:small_subset}
Let $F$ be a finite subset of an infinite group $G$ such that $F=F^{-1}$ and $e\notin F$.
Then there exists a subset $A\subset G$ such that $e\in A$ and
\[
A^{-1}A=G\setminus F.
\]
\end{lemma}

\begin{proof}
If $F=\varnothing$, take $A=G$. Assume $F\neq\varnothing$. 
Let $\kappa=|G|$ and fix an injective enumeration $G\setminus F=\{g_\xi:\xi<\kappa\}$.
We construct by transfinite induction elements $a_\xi\in G$ such that for each $\xi<\kappa$ the set
\[
A_\xi=\{a_\eta,\ a_\eta g_\eta\mid\eta<\xi\}
\]
satisfies
\begin{equation}\label{eq:A_cap_FA}
  A_\xi\cap A_\xi F=\varnothing.
\end{equation}
Set $A_0=\varnothing$ and $a_0=e$, so $A_1=\{e,g_0\}$. 
One easily checks that \eqref{eq:A_cap_FA} holds for $\xi=1$.
Indeed, $A_1F=F\cup g_0F$. Since $e\notin F$ and $g_0\notin F$ (as $g_0\in G\setminus F$), 
and since $F=F^{-1}$ implies $e\notin g_0F$ (equivalently $g_0^{-1}\notin F$), 
we get $A_1\cap A_1F=\varnothing$.

Assume \eqref{eq:A_cap_FA} holds for some $\xi<\kappa$ and $\xi\geq1$. 
Since $|A_\xi|<|G|$ and $F$ is finite, 
we have
\[
A_\xi F\cup A_\xi Fg_\xi^{-1}\neq G,
\]
so we can choose
\[
a_\xi\in G\setminus(A_\xi F\cup A_\xi Fg_\xi^{-1}).
\]
With this choice of $a_\xi$, and using $e\notin F$, $g_\xi\notin F$, and $F=F^{-1}$,
neither $a_\xi$ nor $a_\xi g_\xi$ belongs to 
\[
A_\xi F\ \cup\ a_\xi F\ \cup\ a_\xi g_\xi F.
\]
Hence \eqref{eq:A_cap_FA} holds for $A_{\xi+1}=A_\xi\cup\{a_\xi,\, a_\xi g_\xi\}$ and
\[
g_\xi=a_\xi^{-1}(a_\xi g_\xi)\in A_{\xi+1}^{-1}A_{\xi+1}.
\]
At limit ordinals, we take unions. Finally set
\[
A =\bigcup_{\xi<\kappa}\{a_\xi,\ a_\xi g_\xi\}.
\]
Then $G\setminus F\subset A^{-1}A$. Moreover, $A\cap AF=\varnothing$ implies $F\cap A^{-1}A=\varnothing$, 
hence $A^{-1}A\subset G\setminus F$. Therefore $A^{-1}A=G\setminus F$.
\end{proof}

\begin{lemma}\label{lemma:finite_F_reduction_to_Delta}
Let $G$ be an infinite group, and let $F\subset G$ be a finite symmetric subset with $e\notin F$.
Then there exists $A\subset G$ such that, for $\Gamma=\Cay(G,\partial A)$, one has:
$\Gamma$ is well-covered if and only if all maximal cliques of $\Delta(F)$ have the same cardinality.
\end{lemma}

\begin{proof}
Immediate from Lemmas~\ref{lemma:small_subset} and \ref{lemma:DeltaF_cliques}.
\end{proof}

\section{Finite configurations forcing instability}

\begin{lemma}\label{lemma:on_subset_F}
Let $G$ be a group (finite or infinite). For each set $F$ appearing below, let $\Delta=\Delta(F)$.
\begin{enumerate}[font=\upshape,label=(\alph*)]

\item \label{lemma:on_subset_F_invol}
If $H\leq G$ is a finite subgroup, $|H|>2$, $f\in G\setminus H$, $f^2=e$, and
\[
F=\{f\}\cup H^*,
\]
then $f$ is an isolated vertex of $\Delta$, hence $\iota(\Delta)=1$, and
\[
\omega(\Delta)\ge |H|-1.
\]

\item \label{lemma:on_subset_F_noncub}
If $H\leq G$ is a finite subgroup, $|H|>2$, $f\in G\setminus H$, $f^2\notin H$, $f^3\neq e$, and
\[
F=\{f,f^{-1}\}\cup H^*,
\]
then $f$ is an isolated vertex of $\Delta$, hence $\iota(\Delta)=1$, and
\[
\omega(\Delta)\ge |H|-1.
\]

\item \label{lemma:on_subset_F_cub}
If $H\leq G$ is a finite subgroup with $|H|>3$, $f\in G\setminus H$, $f^3=e$, and
\[
F=\{f,f^{-1}\}\cup H^*,
\]
then
\[
\iota(\Delta)=2
\quad\text{and}\quad
\omega(\Delta)\ge |H|-1.
\]

\item \label{lemma:on_subset_F_infinite_cyclic}
If $S=\langle s\rangle\le G$ is an infinite cyclic subgroup, $n>2$,
\[
H=\{s^{\pm i}\mid i=1,\ldots,n\}, \quad m\ge 2n+1,
\]
and
\[
F=\{s^m,s^{-m}\}\cup H,
\]
then $s^m$ is an isolated vertex of $\Delta$, hence $\iota(\Delta)=1$, and
\[
\omega(\Delta)\ge n-1.
\]
\end{enumerate}
\end{lemma}

\begin{proof}
We first prove the lower bounds for $\omega(\Delta)$.

\smallskip
\noindent\textbf{Cases (a)--(c).}
For distinct $h_1,h_2\in H^*$ we have $h_1h_2^{-1}\in H^*\subseteq F$, so $H^*$ is a clique in $\Delta$.
Therefore
\[
\omega(\Delta)\ge |H^*|=|H|-1.
\]

\smallskip
\noindent\textbf{Case (d).}
Let $X=\{s,s^2,\dots,s^{n-1}\}\subseteq H$. For distinct $s^i,s^j\in X$,
\[
s^i(s^j)^{-1}=s^{i-j},\quad 1\le |i-j|\le n-2,
\]
hence $s^{i-j}\in H\subseteq F$. Thus $X$ is a clique in $\Delta$, so
\[
\omega(\Delta)\ge |X|=n-1.
\]

It remains to compute $\iota(\Delta)$, i.e., the minimum size of a maximal clique in $\Delta$.

\medskip
\noindent\textbf{(a)}\;
We show that $f$ is an isolated vertex of $\Delta$.
Let $x\in F\setminus\{f\}=H^*$. Then $xf\notin H$, hence
\[
xf\notin \{f\}\cup H^*=F,
\]
so $x\not\sim f$ in $\Delta$.
Thus $f$ is an isolated vertex, and $\iota(\Delta)=1$.

\medskip
\noindent\textbf{(b)}\;
Again, we show that $f$ is an isolated vertex of $\Delta$.
Let $x\in F\setminus\{f\}$. 
If $x=f^{-1}$ and $f\sim x$, then $f^2\in F$. This is impossible, since $f^2$ can be neither $f$ nor $f^{-1}$
(otherwise $f^3=e$), and also $f^2\notin H^*$.

If $x\in H^*$ and $f\sim x$, then $fx^{-1}\in F$. This is impossible, since $fx^{-1}$ can be neither $f$ nor $f^{-1}$
(otherwise $f^2\in H$), and also $fx^{-1}\notin H^*$.

Thus $f$ is adjacent to no element of $F\setminus\{f\}$, so $f$ is an isolated vertex. 
Hence $\iota(\Delta)=1$.

\medskip
\noindent\textbf{(c)}\;
Set $T:=\{f,f^{-1}\}$. This is a clique in $\Delta$, since
\[
f(f^{-1})^{-1}=f^2=f^{-1}\in F.
\]
We show that $T$ is maximal. Let $x\in F\setminus T=H^*$.

Neither $xf$ nor $xf^{-1}$ lies in $H$ (otherwise $f\in H$).
Also $xf,xf^{-1}\notin\{f,f^{-1}\}$: indeed, $xf=f$ or $xf^{-1}=f^{-1}$ gives $x=e$,
while $xf=f^{-1}$ or $xf^{-1}=f$ gives $f^2\in H$, hence $f\in H$, a contradiction.
Thus $xf,xf^{-1}\notin F$, so $x$ is not adjacent to either $f$ or $f^{-1}$.
Therefore $T$ is a maximal clique, and
\[
\iota(\Delta)\le 2.
\]

No vertex of $\Delta$ is isolated, since for every $x\in F$ the set $\{x\}$ is contained in a $2$-clique:
if $x\in\{f,f^{-1}\}$, then $\{x\}\subset T$; if $x\in H^*$, then (because $|H|\ge 3$) there exists
$y\in H^*$ with $y\ne x$, and then $xy^{-1}\in H^*\subseteq F$, so $\{x,y\}$ is a clique in $\Delta$.
Hence $\iota(\Delta)\ge 2$, and therefore
\[
\iota(\Delta)=2.
\]

\medskip
\noindent\textbf{(d)}\;
We show that $s^m$ is an isolated vertex of $\Delta$.
Let $x\in F\setminus\{s^m\}$. Write
\[
x=s^k,\quad k\in\{\pm1,\dots,\pm n,-m\}.
\]
We show that $xs^{-m}\notin F$.
If $k=-m$, then $xs^{-m}=s^{-2m}\notin F$.
If $1\le |k|\le n$, then $k-m<-n$ and $k-m\ne -m$, so $s^{k-m}\notin F$. 
Thus $x$ is not adjacent to $s^m$, and $s^m$ is an isolated vertex. 
Hence $\iota(\Delta)=1$.
\end{proof}

\noindent
\smallskip
\textbf{Remark.}
In each part of Lemma~\ref{lemma:on_subset_F}, the assumptions on $|H|$  
are chosen so that the stated bounds imply 
$\iota(\Delta)<\omega(\Delta)$
Indeed, in \ref{lemma:on_subset_F_invol}, \ref{lemma:on_subset_F_noncub}, and
\ref{lemma:on_subset_F_infinite_cyclic} we have $\iota(\Delta)=1$, while
$\omega(\Delta)\ge |H|-1\geq2$ (respectively, $\omega(\Delta)\ge n-1\geq2$).
In \ref{lemma:on_subset_F_cub} we have $\iota(\Delta)=2$ and $\omega(\Delta)\ge |H|-1>2$.

\section{Auxiliary lemmas on 2-groups}

\begin{lemma}\label{lemma:G/H_elementary}
Let $G$ be an infinite group and $H\leq G$ a finite subgroup. Assume that $g^2\in H$ for every $g\in G$.
Then $H\trianglelefteq G$ and $G/H$ is abelian of exponent $2$.
\end{lemma}

\begin{proof}
Fix $g\in G$ and $h\in H$. By assumption, $(gh)^2\in H$, and also $g^2\in H$, hence
\[
(gh)^2(g^2)^{-1}=g h g h g^{-2}=g h g^{-1}\in H.
\]
Thus $gHg^{-1}\subset H$ for all $g$, so $H\trianglelefteq G$.
Therefore $G/H$ is abelian of exponent $2$.
\end{proof}

\begin{lemma}
\label{lemma:32}
There is no group $G$ of order $32$ such that
\begin{enumerate}
\item $G$ has exactly one involution, and
\item $S:=\{g^2\mid g\in G\}$ satisfies $|S|\leq 4$.
\end{enumerate}
\end{lemma}

\begin{proof}
By \cite[5.3.6]{Robinson}, a finite $2$-group with a unique involution
is either cyclic or generalized quaternion. Hence $G\cong C_{32}$ or $G\cong Q_{32}$.

If $G\cong C_{32}=\langle x\rangle$, then
$S=\{x^{2k}\mid 0\leq k<32\}=\langle x^2\rangle$, so $|S|=16>4$.

If $G\cong Q_{32}=\langle a,b\mid a^{16}=1,\ b^2=a^8,\ bab^{-1}=a^{-1}\rangle$, then every
element is $a^k$ or $a^k b$, and
$(a^k)^2=a^{2k}$ while $(a^k b)^2=a^k b a^k b=b^2=a^8$.
Thus $S=\langle a^2\rangle$, so $|S|=8>4$.

In both cases $|S|>4$, a contradiction.
\end{proof}

\begin{corollary}\label{corollary:uniq_invol}
Let $G$ be a $2$-group and let $H=\langle h\rangle$ have order $4$. Assume that
\[
\forall g\in G\setminus H\quad g^2\in\{h,h^2,h^3\}.
\]
Then $G$ is a finite group.
\end{corollary}

\begin{proof}
Assume, to the contrary, that $G$ is infinite. 
Since $g^2\in H$ for every $g\in G$, we have 
by Lemma~\ref{lemma:G/H_elementary}, $G/H$ is abelian of exponent $2$. Hence every finitely generated subgroup
of $G$ containing $H$ is finite, so $G$ is locally finite.

Therefore $G$ contains a finite subgroup $K\le G$ with $H\leq K$ and $|K|\ge 32$. Replacing $K$ by the full
preimage of a subgroup of order $8$ in $K/H$, we may assume that $|K|=32$ (and still $H\le K$).

Then $K$ satisfies the same square condition:
\[
\forall x\in K\setminus H\quad x^2\in\{h,h^2,h^3\}.
\]
Hence $\{x^2\mid x\in K\}\subseteq H$, so $K$ has at most four squares. Also $K$ has exactly one involution,
namely $h^2$, since for $x\in K\setminus H$ we have $x^2\ne 1$.
This contradicts Lemma~\ref{lemma:32}. Therefore $G$ is finite.
\end{proof}

\section{Proof of the theorem}

\begin{theorem*}
Every infinite group is unstable.
\end{theorem*}

\begin{proof}
Let $G$ be an infinite group.
By Lemma~\ref{lemma:finite_F_reduction_to_Delta}, 
it is enough to construct a finite symmetric set $F\subset G$ with $e\notin F$
such that, for $\Delta=\Delta(F)$, one has 
\begin{equation}\label{iota<omega}
  \iota(\Delta)<\omega(\Delta).
\end{equation}
We distinguish cases according to the algebraic structure of $G$.
 
\medskip
\noindent\textbf{Case 1:} $G$ contains an element of infinite order.
Then Lemma~\ref{lemma:on_subset_F}\ref{lemma:on_subset_F_infinite_cyclic} applies.

\medskip
\noindent\textbf{Case 2:} every element of $G$ has finite order.

\smallskip
\noindent\textbf{Case 2A}: $G$ has a nontrivial element of odd order.
Choose a nontrivial element $h\in G$ of odd order, and let $H=\langle h\rangle$.

If there exists $f\in G\setminus H$ such that $f^2\in H$, 
then the cyclic subgroup $\langle f\rangle$ has order $2m$, where $m>1$ is odd, and
\[
\langle f\rangle=\langle f'\rangle\times H',
\]
with $f'^2=e$ and $|H'|=m$. 
Thus, taking $f'$ in place of $f$ and $H'$ in place of $H$, 
we are in the setting of 
Lemma~\ref{lemma:on_subset_F}\ref{lemma:on_subset_F_invol}.

Otherwise, for every $x\in G\setminus H$ we have $x^2\notin H$.

If there exists $f\in G\setminus H$ with $f^3\neq e$, then
Lemma~\ref{lemma:on_subset_F}\ref{lemma:on_subset_F_noncub} applies.

If no such element exists, then every element of $G\setminus H$ has order $3$.
In particular, for every $f\in G\setminus H$ we have $f^3=e$ and
$f^2=f^{-1}\notin H$.
If, moreover, $|H|>3$, then
Lemma~\ref{lemma:on_subset_F}\ref{lemma:on_subset_F_cub} applies.

If $|H|=3$ and every element of $G$ has order $3$, then $G$ is a locally finite group
(see \cite[14.2.3]{Robinson}). Hence $G$ contains a finite subgroup $K$ with $|K|>3$.
Since $G$ is infinite, choose $f\in G\setminus K$. Then $f^3=e$, and
Lemma~\ref{lemma:on_subset_F}\ref{lemma:on_subset_F_cub} applies (with $H=K$).

\smallskip
\noindent\textbf{Case 2B:} every nontrivial element of $G$ has even order.
Then $G$ is a $2$-group.

If every nontrivial element of $G$ has order $2$, 
then Lemma~\ref{lemma:on_subset_F}\ref{lemma:on_subset_F_invol} applies.

Thus $G$ contains an element of order greater than $2$, and hence (taking a suitable power) an element of order $4$.
Let $H<G$ be a cyclic subgroup of order $4$.

If there exists $f\in G\setminus H$ such that $f^2\notin H$, 
then Lemma~\ref{lemma:on_subset_F}\ref{lemma:on_subset_F_noncub} applies.

If there exists $f\in G\setminus H$ such that $f^2=e$, 
then Lemma~\ref{lemma:on_subset_F}\ref{lemma:on_subset_F_invol} applies.

If for every $f\in G\setminus H$ we have $f^2\neq e$ and $f^2\in H$, then it follows from Corollary~\ref{corollary:uniq_invol} that $G$ is finite, a contradiction.
Thus, in all cases, there exists a finite symmetric set $F\subset G$ with $e\notin F$ 
such that \eqref{iota<omega} holds.
Therefore $G$ is unstable.
\end{proof}

\end{document}